\documentclass{amsart}
\usepackage{hyperref}
\usepackage{amsthm}
\usepackage{amssymb}
\usepackage{amsmath}
\usepackage{amscd}
\usepackage{cases}
\usepackage[dvipdfmx]{graphicx}
\usepackage{latexsym}
\usepackage{pifont}

\newtheorem{thm}{Theorem}[section]
\newtheorem{cor}[thm]{Corollary}
\newtheorem{prop}[thm]{Proposition}
\newtheorem{lem}[thm]{Lemma}

\theoremstyle{definition}
\newtheorem{deff}[thm]{Definition}

\newtheorem{pr}[thm]{Proof}

\newcommand{\R}{\mathbb{R}}
\newcommand{\tr}{\mathrm{tr}}
\makeatletter
 
 \@addtoreset{equation}{section}
\makeatother

\begin{document}

\title{{Trivializing number of positive knots}}
\author{Kazuhiko Inoue}
\address{Graduate School of Mathematics, Kyusyu University, 744, Motooka, Nishi-ku, Fukuoka, 819-0395, Japan}
\date{\today}

\maketitle

\begin{abstract}
In this paper, we give the trivializing number of all minimal diagrams of positive 2-bridge knots and study the relation between the trivializing number and the unknotting number for a part of these knots.
\end{abstract}

\section{Introduction}

The trivializing number is one of numerical invariants of knots as same as the unknotting number, which judges some complexity of a knot. In general, it is known that the former is equal or over than twice the latter. Furthermore, Hanaki has conjectured that the trivializing number of a positive knot is just twice the unknotting number of the same knot. Indeed, for any positive knot up to 10 crossings, the  equality such that $\tr(K) = 2u(K)$ holds, where $tr(K)$ is the trivializing number of a knot and $u(K)$ is the unknotting number of the same knot (see \cite{Ha}), and our result gives a partial positive answer of this conjecture. 

This paper is constructed as follows:\ In Section 2, we define the trivializing number of a diagram and the trivializing number of a knot. In Section 3, we shortly do a review of positive knot and 2-bridge knot. In Section 4, we determine the standard diagrams of positive 2-bridge knots, and in Section 5, we determine the trivializing number of minimal diagrams of positive 2-bridge knots. In section 6, we show that for some positive 2-bridge knots, the equation $tr(K) = 2u(K)$ holds. In Section 7, we introduce positive pretzel knots, and show that the equation above also holds for some of them.

\section{Preliminaries}

We work in the PL category. Throughout this paper, all knots are oriented. A  \emph{projection} of a knot $K$ in $\R^3$ is a regular projection image of $K$ in $\R^2\cup\{\infty\} = S^2$.  A \emph{diagram} of $K$ is a projection endowed with  over/under information for its double points. A \emph{crossing} is a double point with over/under information, and a \emph{pre-crossing} is a double point without over/under information. A \emph{pseudo-diagram} of $K$ is a projection of $K$ whose double points are either crossings or pre-crossings. See Figure \ref{fig:1-01ji}.

 \begin{figure}[htbp]
  \centering
  \includegraphics[width=10cm]{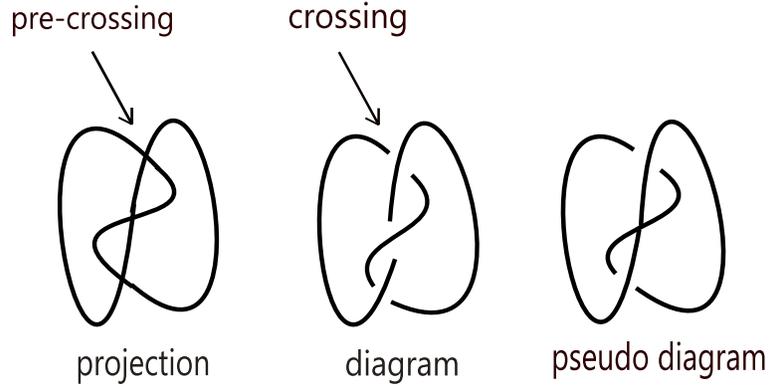}
  \caption{Projection, diagram, and pseudo-diagram}
  \label{fig:1-01ji}
 \end{figure}
 
A pseudo-diagram is said to be \emph{trivial} if we always get a diagram of a trivial knot after giving arbitrary over/under information to all the pre-crossings. An example is given in Figure \ref{fig:1-02ji}. It is known that we can change every projection into a trivial pseudo-diagram by giving appropriate over/under information to some of the pre-crossings.

 \begin{figure}[htbp]
 \centering
 \includegraphics[width=9cm]{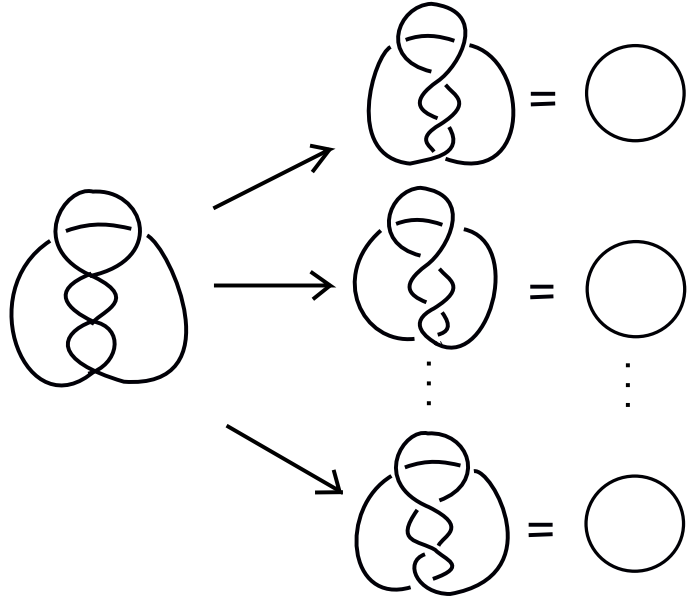}
 \caption{Example of a trivial pseudo-diagram}
 \label{fig:1-02ji}
 \end{figure}

\begin{deff}\label{def111}
The \emph{trivializing number} of a projection $P$, denoted by ${\tr}(P)$, is the minimal number of pre-crossings of $P$ which should be transformed into crossings for getting a trivial pseudo-diagram.
\end{deff}

\begin{deff}\label{def112}
The \emph{trivializing number} of a diagram $D$, denoted by ${\tr}(D)$, is by definition the trivializing number of the associated projection which is obtained from $D$ by ignoring the over/under information.
\end{deff}

For example in the case as shown in Figure \ref{fig:1-03ji}, we can get a trivial pseudo-diagram by transforming two pre-crossings $\textcircled{\scriptsize1}$ and $\textcircled{\scriptsize2}$ of $P$ into crossings; however, it can be easily checked that we cannot get a trivial pseudo-diagram by transforming only one pre-crossing of $P$ into a crossing. Therefore, we have $\tr(D)=\tr(P)=2$.

\begin{figure}[htbp]
  \centering
  \includegraphics[width=10cm]{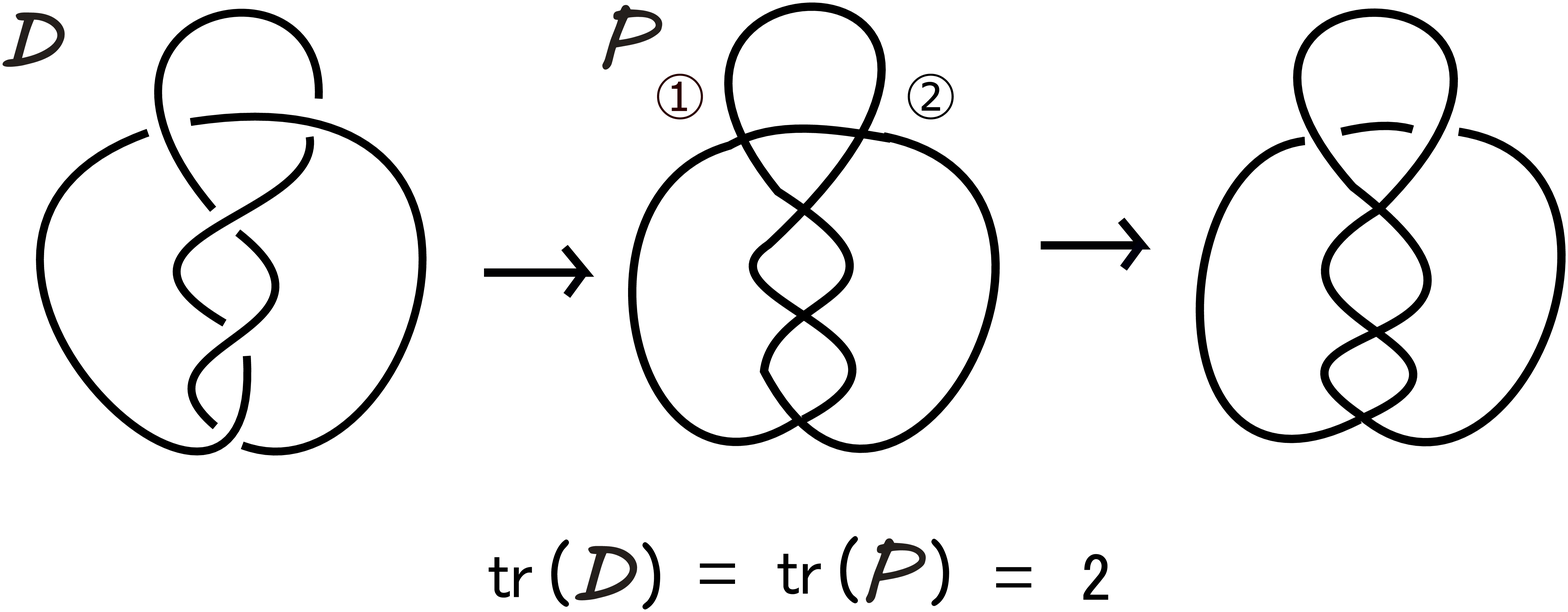}
  \caption{An operation for getting a trivial pseudo-diagram}
  \label{fig:1-03ji}
 \end{figure}

\begin{deff}\label{def113}
The \emph{trivializing number} of a knot $K$, denoted by ${\tr}(K)$, is the minimum of  ${\tr}(D)$, where the minimum is taken over all diagrams $D$ of $K$.
\end{deff}

\section{Positive 2-bridge knots}

Generally speaking, the trivializing number of a knot is not always realized by its minimal diagram (a diagram that has the minimal number of crossings); in fact we have counter examples (see  \cite{Ha}). Moreover, even for a given diagram, determining its trivializing number is not so easy in general. In Section 5, we give the trivializing numbers of all minimal diagrams of positive 2-bridge knots.

Let $D$ be an oriented diagram of a knot. To each of its crossings, we associate sign $+$ or $-$ as shown in Figure \ref{fig:2-01ji}(1). If all the crossings in $D$ have the same sign $+$ (resp.\ $-$), then we say that $D$ is a \emph{positive diagram} (resp.\ \emph{negative diagram}).

\begin{figure}[htbp]
 \centering
 \includegraphics[width=9cm,clip]{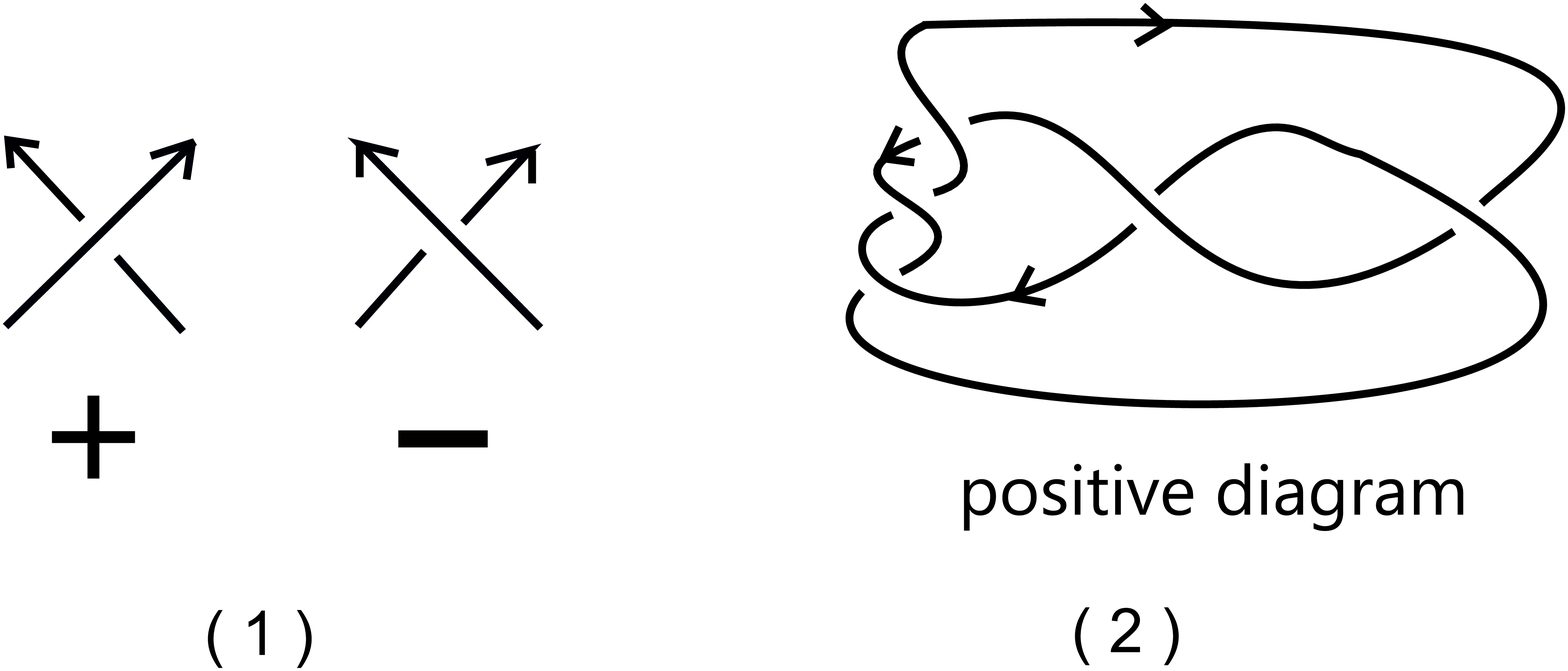}
 \caption{Sign of a crossing and an example of a positive diagram}
 \label{fig:2-01ji}
\end{figure}

When $D$ is a positive diagram, the mirror image of $D$, which is obtained by changing the over/under information of all crossings of $D$ and is denoted by $D^*$, is a negative diagram. Since $D$ and $D^*$ correspond to the same projection, we have $\tr(D) = \tr(D^*)$. A \emph{positive knot} is a knot which has a positive diagram.

For a finite sequence $a_1, a_2, \ldots, a_m$ of integers, let us consider the knot (or link) diagram $D(a_1, a_2, \ldots, a_m)$ as shown in Figure \ref{fig:2-02ji}. In the figure, a rectangle in the upper row (resp.\ lower row), depicted by double lines (resp.\ simple lines), with integer $a$ represents $a$ left-hand (resp.\ right-hand) horizontal half-twists if $a \geq 0$, and $|a|$ right-hand (resp.\ left-hand) horizontal half-twists if $a < 0$. See Figure \ref{fig:2-03ji} for some explicit examples. We say a rectangle in the upper row (resp.\ lower row), an upper rectangle (resp.\ a lower rectangle) for short. A knot which is represented by such a diagram is called a \emph{$2$-bridge knot}.

\begin{figure}[htbp]
 \centering
 \includegraphics[width=10cm, clip]{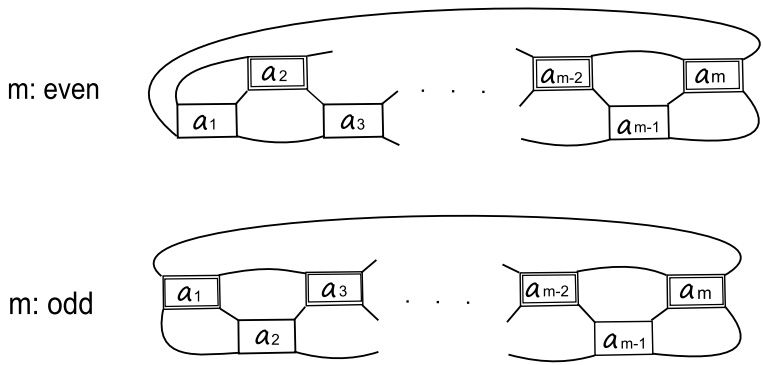}
 \caption{2-bridge knot diagrams}
 \label{fig:2-02ji}
\end{figure}

\begin{figure}[htbp]
 \centering
 \includegraphics[width=12cm, clip]{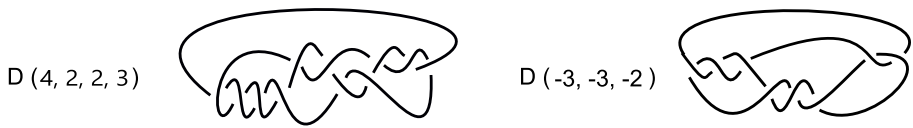}
 \caption{Examples of 2-bridge knot diagrams}
 \label{fig:2-03ji}
\end{figure}
 
If  $a_i > 0$ for all $i$ with $1 \leq i \leq m$ or if $a_i < 0$ for all $i$, then the diagram $D(a_1, a_2, \ldots, a_m)$ is reduced and alternating, and hence is a minimal diagram (see \cite{Mu}). We call such a diagram a \emph{standard diagram} of the knot.

It is known that every $2$-bridge knot has a unique standard diagram (see, for example, \cite{Mu}). Therefore, a positive (resp.\ negative) $2$-bridge knot is a positive (resp.\ negative) alternating knot. A positive alternating knot may not necessarily have a diagram which is both positive and alternating in general. However,  Nakamura has shown the following.

\begin{thm}[Nakamura \cite{Na}] \label{thm114}

 A reduced alternating diagram of a positive alternating knot is positive.

\end{thm}

By the theorem abobe, the standard diagram of a positive 2-bridge knot is necessarily positive.

 In order to study the trivializing number of the standard diagram $D$ of a positive or negative 2-bridge knot, by taking the mirror image, we may assume $a_i > 0$ for all $i$. Note that a positive knot may turn into a negative one by this operation. 
 
\section{standard diagrams of positive 2-bridge knots}
 
In this section, we determine the standard diagrams of positive 2-bridge knots.

\begin{prop}\label{prop116}

Let $D = D(a_1, a_2, \ldots , a_m)$ be a positive diagram or a negative diagram of a $2$-bridge knot such that $a_i > 0$ for all $i$ with $1 \leq i \leq m$. Then $D$ must be one of the following forms.

\begin{enumerate}
 \item When $m$ is even, say $m = 2n$, we have either
  \begin{enumerate}
  \item[(a1)] $a_{2i}$ is even for $1 \leq i \leq n-1$,
  \item[(a2)] $a_{2n}$ is odd, and
  \item[(a3)] $\sum_{i=1}^{n} a_{2i-1}$ is even,
  \end{enumerate}
  or    
  \begin{enumerate}
  \item[(b1)] $a_1$ is odd,
  \item[(b2)] $a_{2i-1}$  is even for $2 \leq i \leq n$, and
  \item[(b3)] $\sum_{i=1}^{n} a_{2i}$ is even.
  \end{enumerate}
  
 \item When $m$ is odd, say $m = 2n + 1$, we have either
  \begin{enumerate}
  \item[(a1)] $a_{2i-1}$ is even for $2 \leq i \leq n$,
  \item[(a2)] $a_1$ and $a_{2n+1}$ are odd, and
  \item[(a3)] $\sum_{i=1}^{n} a_{2i}$ is odd,
  \end{enumerate}
  or    
  \begin{enumerate}
  \item[(b1)] $a_{2i}$ is even for $1\leq i \leq n$, and
  \item[(b2)] $\sum_{i=1}^{n+1} a_{2i-1}$ is odd.
  \end{enumerate}

\end{enumerate}

\end{prop}

Let us consider a rectangle with integer $a_i > 0$, as in Figure \ref{fig:2-02ji}, which corresponds to $a_i$ left-hand (resp.\ right-hand) half-twists if it is in the upper (resp.\ lower) row. In the following, such a rectangle will sometimes be denoted by $(a_i)$. If its crossings all have the same sign $+$, then the orientation of the two arcs are of a form as in Figure \ref{fig:3-01ji}. If the crossings all have sign $-$, then they are of a form as in Figure \ref{fig:3-02ji}. Furthermore, we adopt the symbolic convention as depicted in Figure \ref{fig:3-02-01ji}.

\begin{figure}[htbp]
 \centering
 \includegraphics[width=12cm,clip]{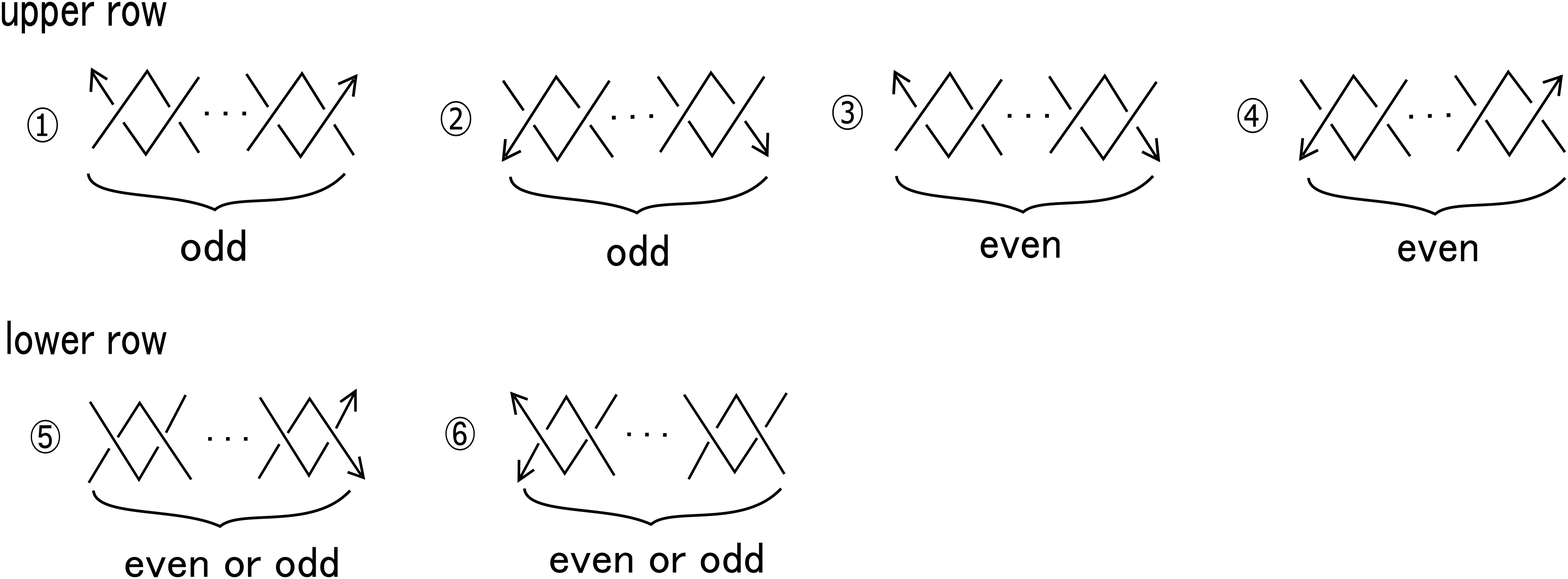}
 \caption{The orientations of two arcs with positive crossings}
 \label{fig:3-01ji}
 \end{figure}

\begin{figure}[htbp]
 \centering
 \includegraphics[width=12cm,clip]{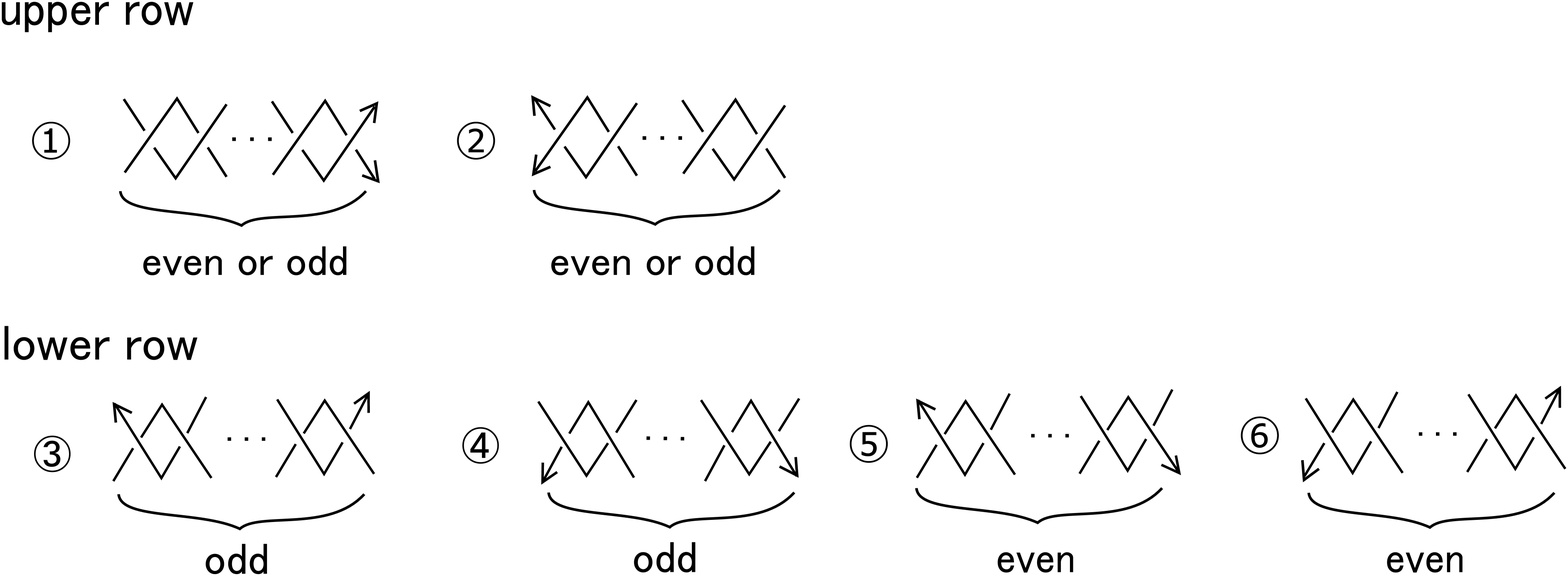}
 \caption{The orientations of two arcs with negative crossings}
 \label{fig:3-02ji}
\end{figure}

\begin{figure}[htbp]
 \centering
 \includegraphics[width=9cm,clip]{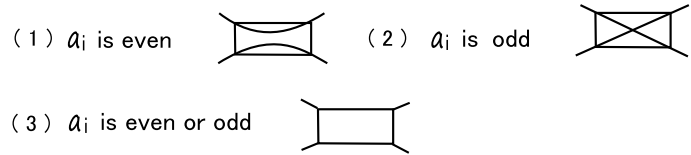}
 \caption{The symbolic convention}
 \label{fig:3-02-01ji}
\end{figure}

\begin{proof}[Proof of Proposition~ \textup{\ref{prop116}}]

(1) We may assume that the orientation of the diagram is as depicted in Figure \ref{fig:3-03ji}, since it is a diagram of a knot.

\begin{figure}[htbp]
 \centering
 \includegraphics[width=8cm,clip]{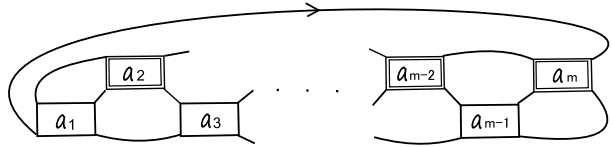}
 \caption{Orientation of diagram $D = D(a_1, a_2, \ldots, a_m)$, $m$ is even.}
 \label{fig:3-03ji}
\end{figure}

When the crossings in $(a_i)$ with $1 \leq i \leq 2n$ all have the same sign $+$, the orientations of the two arcs of $(a_{2n})$ are as shown in Figure \ref{fig:3-04ji} $\textcircled{\scriptsize1}$ or $\textcircled{\scriptsize2}$. Then the orientation of the diagram is as shown in Figure \ref{fig:3-04ji} $\textcircled{\scriptsize4}$. Since the orientations of the arcs of $(a_{2n-1})$ must be as shown in Figure \ref{fig:3-04ji} $\textcircled{\scriptsize3}$, the orientations of the arcs of $(a_{2n})$ must be as in Figure \ref{fig:3-04ji} $\textcircled{\scriptsize1}$. In particular, $a_{2n}$ is necessarily odd. Due to the orientation of $(a_{2n-1})$, we see that the orientation of the other $a_{2i-1}$, $1 \leq i \leq n$ are of the form as in Figure \ref{fig:3-04ji} $\textcircled{\scriptsize3}$. Furthermore, by chasing the oriented arcs, we can determine the orientations of all arcs in the remaining rectangles. Hence, the oriented diagram is as depicted in Figure \ref{fig:3-05ji}.

\begin{figure}[htbp]
 \centering
 \includegraphics[width=8cm,clip]{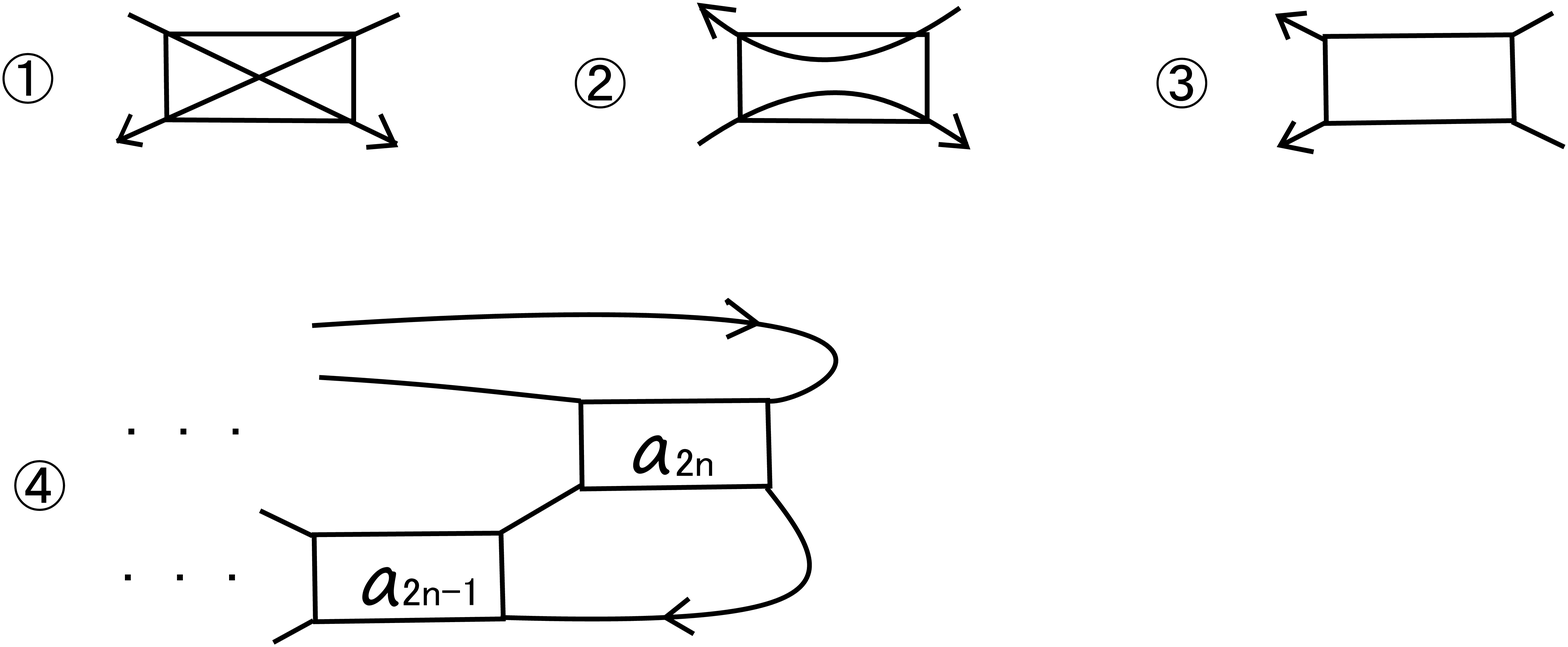}
 \caption{The orientation of arcs of rectangles}
 \label{fig:3-04ji}
\end{figure}

\begin{figure}[htbp]
 \centering
 \includegraphics[width=10cm,clip]{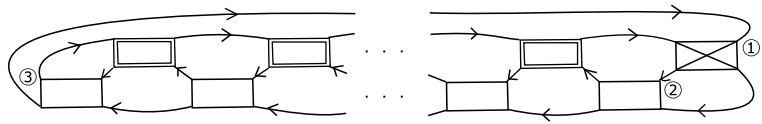}
 \caption{Orientations of arcs in $D$}
 \label{fig:3-05ji}
\end{figure}

Then, we see that $a_{2i}$, $1 \leq i \leq n-1$, are all even. Since this is a knot diagram, the oriented strand passing through $\textcircled{\scriptsize1}$ and then $\textcircled{\scriptsize2}$ needs to pass through $\textcircled{\scriptsize3}$. Therefore, $\sum_{i=1}^n a_{2i-1}$ must necessarily be even. This shows that the conditions $(a1)$, $(a2)$ and $(a3)$ are satisfied.

On the other hand, when the signs of crossings in $(a_i)$ are all $-$, by turning the diagram as shown in Figure \ref{fig:3-05-01ji}(1) on the plane by 180 degrees, and by reversing orientaitions of all arcs, we can get the mirror image of the diagram with the sign $+$ as shown in Figure \ref{fig:3-05-01ji}(2). Consequently, we know that the conditions $(b1)$, $(b2)$ and $(b3)$ are satisfied.

\begin{figure}[htbp]
 \centering
 \includegraphics[width=9cm,clip]{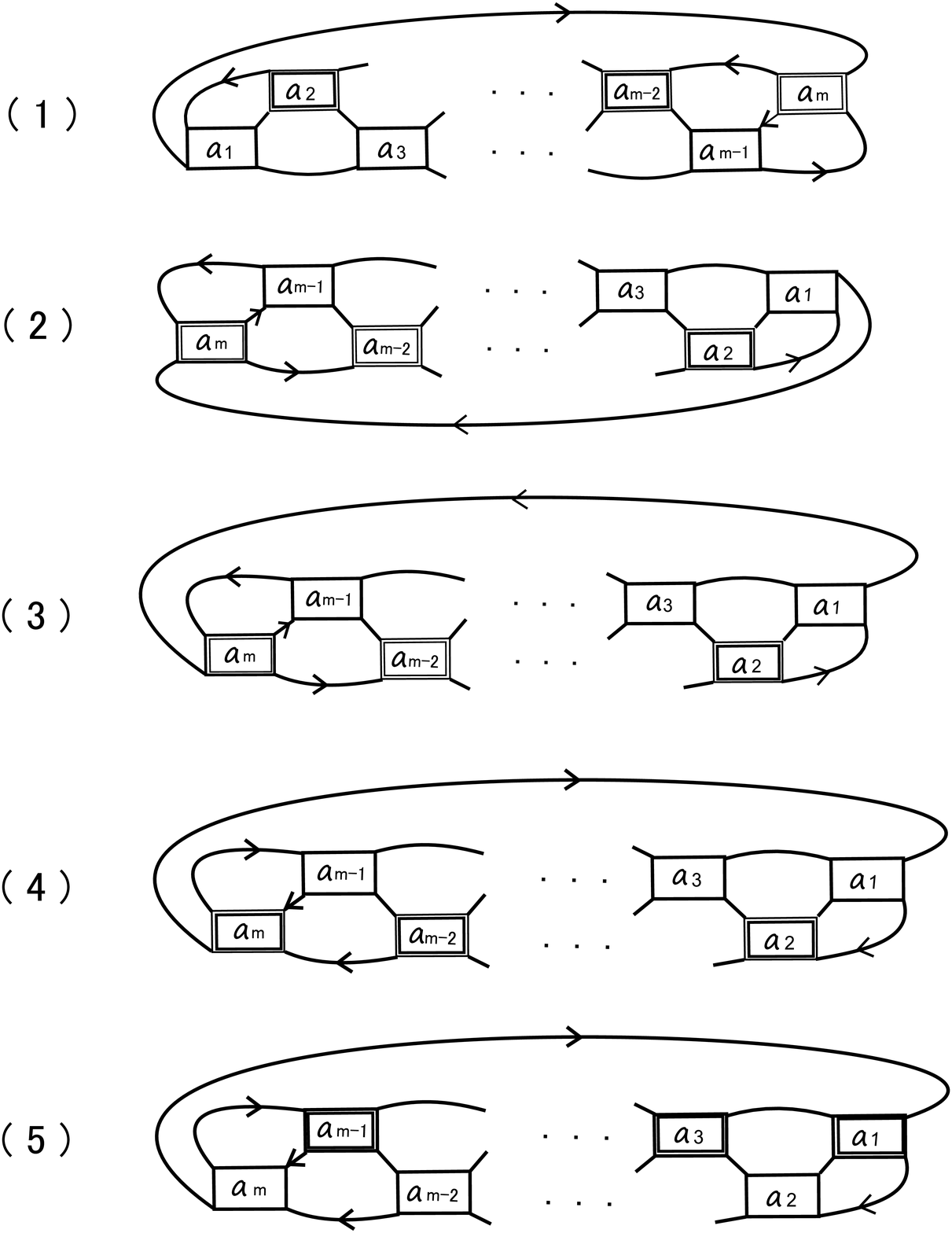}
 \caption{Orientations of arcs in $D$ with all signs $-$}
 \label{fig:3-05-01ji}
\end{figure}

(2)When all crossings in $(a_i)$ with $1 \leq i \leq 2n + 1$ have the same sign $+$, we can consider in a similar way to  the case where $m$ is even. We may assume that the orientation of the diagram is as shown in Figure \ref{fig:3-06ji}.

\begin{figure}[htbp]
 \centering
 \includegraphics[width=8cm,clip]{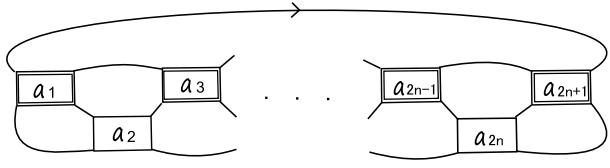}
 \caption{Orientation of diagram $D = D(a_1, a_2, \ldots, a_m)$, $m$ is odd,}
 \label{fig:3-06ji}
\end{figure}

Since the orientations of the two arcs of $(a_{2n+1})$ are as shown in Figure \ref{fig:3-04-01ji} $\textcircled{\scriptsize1}$ or $\textcircled{\scriptsize2}$, the orientation of the diagram is like as Figure \ref{fig:3-04-01ji} $\textcircled{\scriptsize4}$. Furthermore,  the orientation of the arcs of $(a_{2n})$ is as shown Figure \ref{fig:3-04-01ji} $\textcircled{\scriptsize3}$, the orientation of the arcs of $(a_{2n+1})$ must be as shown in Figure \ref{fig:3-04-01ji} $\textcircled{\scriptsize1}$. In particular,  $a_{2n+1}$ is necessarily odd. Due to the orientation of $(a_{2n})$, we see that the orientations other $(a_{2i})$, $1 \leq i \leq n-1$ are of the form as in Figure \ref{fig:3-04-01ji} $\textcircled{\scriptsize3}$. 

\begin{figure}[htbp]
 \centering
 \includegraphics[width=8cm,clip]{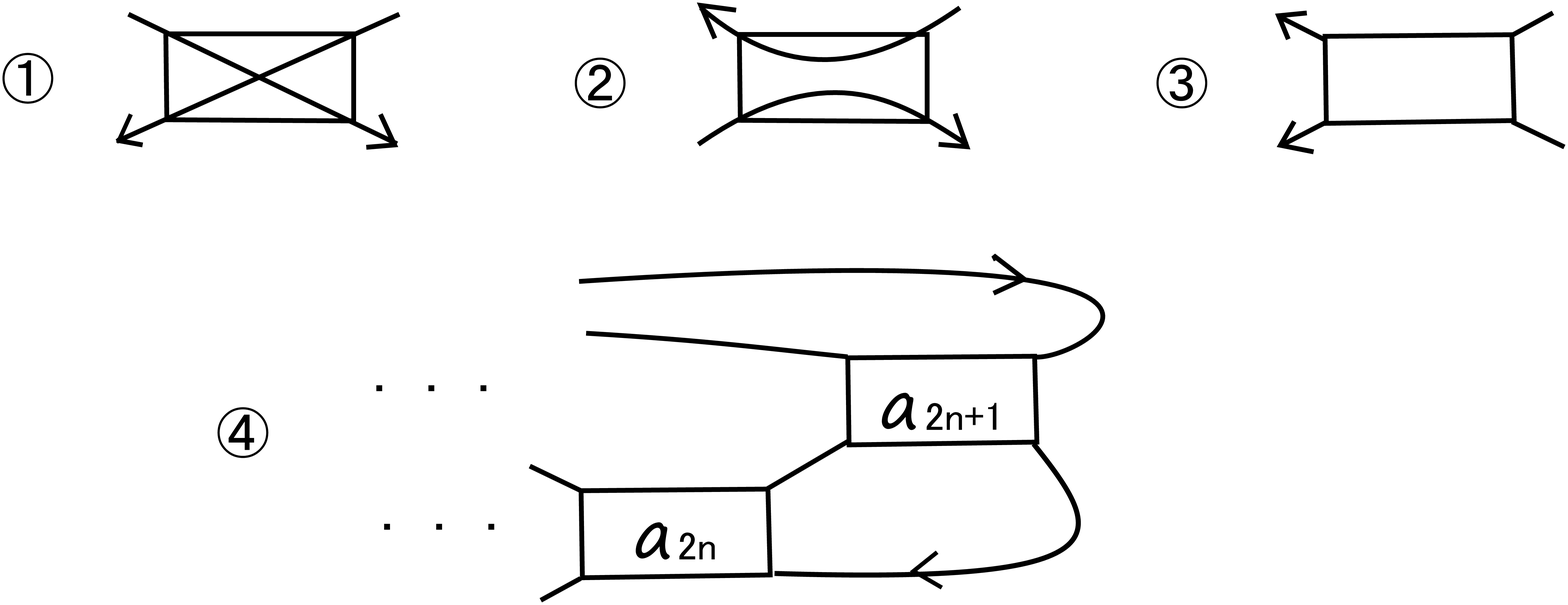}
 \caption{ }
 \label{fig:3-04-01ji}
\end{figure}

Hence, the oriented diagram is as depicted in Figure \ref{fig:3-08ji}. This shows that the conditions $(a1)$, $(a2)$ and $(a3)$ are satisfied. 

\begin{figure}[htbp]
 \centering
 \includegraphics[width=10cm,clip]{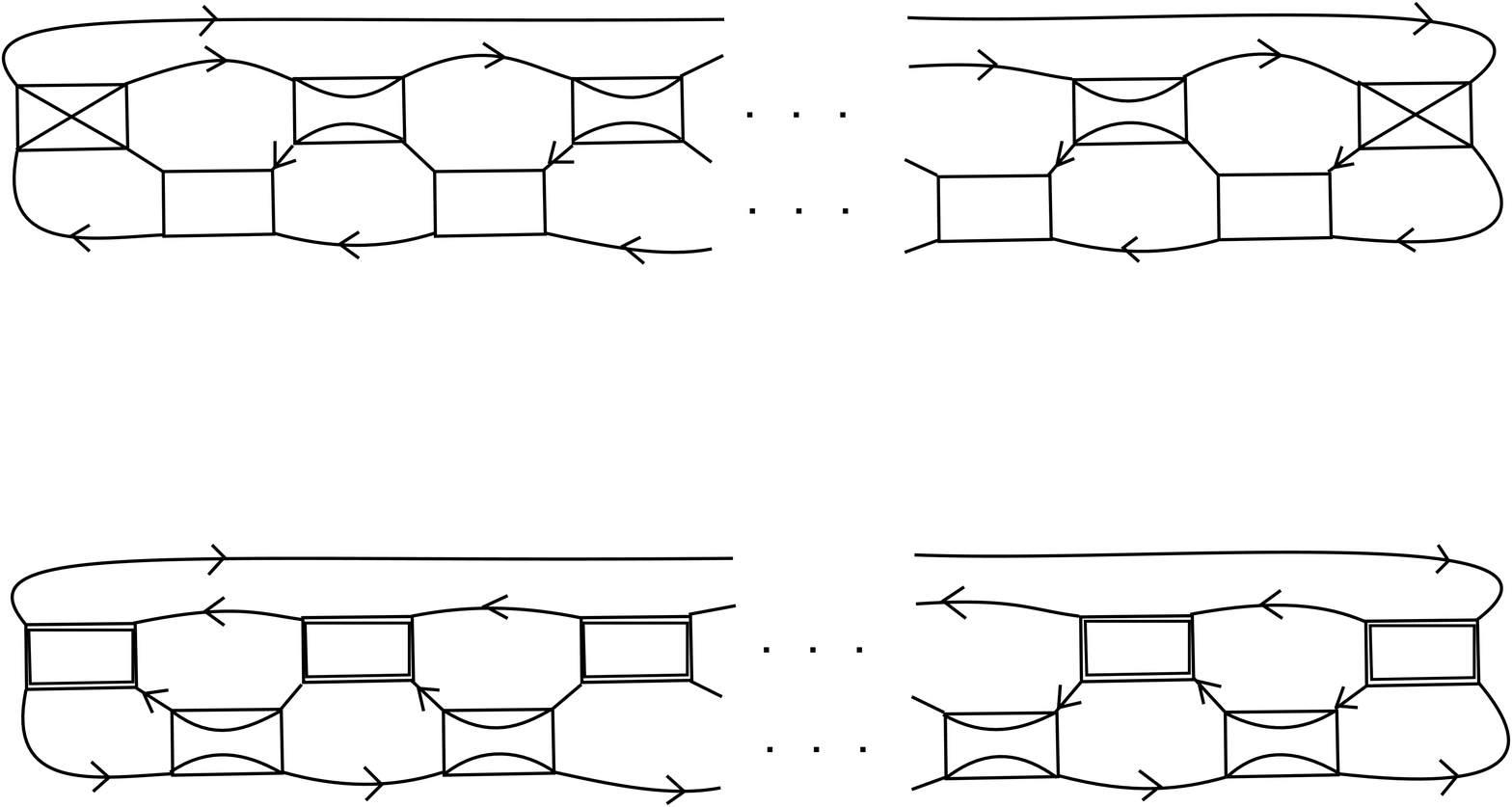}
 \caption{The orientation of a positive diagram $D = D(a_1, a_2, \ldots, a_m)$, $m$ is odd}
 \label{fig:3-08ji}
\end{figure}

If the signs of the crossings in $(a_i)$ are all negative, then the orientation of $a_{2n+1}$ is as shown in Figure \ref{fig:3-08-01ji} $\textcircled{\scriptsize1}$. Therefore, the orientations of the other $a_{2i+1}$ $1geq i \geq n-1$ must be of the same form as shown in Figure \ref{fig:3-08-01ji} $\textcircled{\scriptsize2}$. Thus the diagram is naturally as shown in Figure \ref{fig:3-08-01ji} $\textcircled{\scriptsize3}$, and we can easily see that the conditions $(b1)$ and $(b2)$ are satisfied. This completes the proof.

\end{proof}

\begin{figure}[htbp]
 \centering
 \includegraphics[width=10cm,clip]{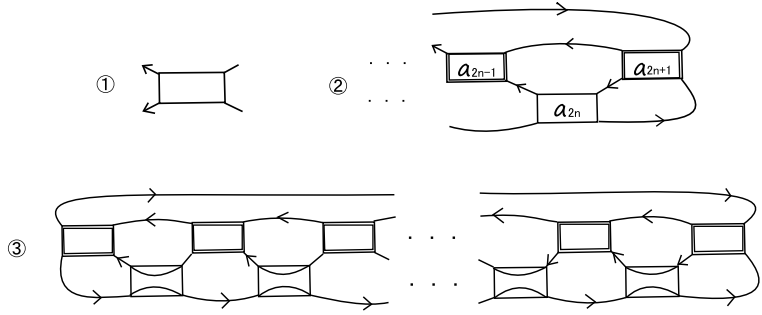}
 \caption{The orientation of a negative diagram $D = D(a_1, a_2, \ldots, a_m)$, $m$ is odd}
 \label{fig:3-08-01ji}
\end{figure}

Beside, we classify these diagrams into four types, that is type of 1a, type of 1b, type of 2a, and type of 2b. Remark that these four types correspond not only to the proposition above but also to Main Theorem. 

\section{Main Theorem}

For determining the trivializing number of a diagram, we can make use of the \emph{chord diagram}. Let $P$ be a projection with $n$ pre-crossings. A \emph{chord diagram} of $P$, denoted by $CD_P$, is a circle with $n$ chords marked on it by dotted line segment where the preimage of each pre-crossing is connected by a chord (see \cite{Ha}). We provide an example of the chord diagram as shown in Figure \ref{fig:3-09ji}.

A chord diagram is said to be parallel if all the chords in it have no intersection. For example, the rightmost chord diagram in Figure \ref{fig:3-09ji} is made of the chords which correspond to the crossings 1, 2, 3, and is parallel.

\begin{figure}[htbp]
 \centering
 \includegraphics[width=9cm,clip]{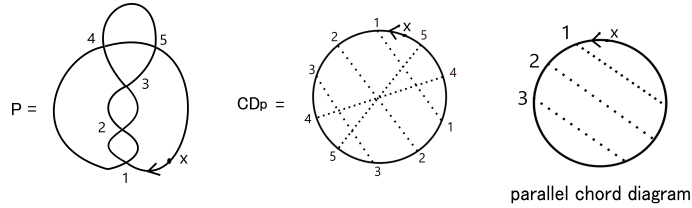}
 \caption{A chord diagram}
 \label{fig:3-09ji}
\end{figure}

In the situation above, next theorem holds.

\begin{thm}[Hanaki \cite{Ha}] \label{thm117}
 
 For a chord diagram, the following holds.
 \begin{itemize}

 \item The number of chords which are taken away from a chord diagram is even.

 \item $tr(D)$ = min \{the number of chords which must be taken away from a chord
 diagram in order to get a parallel chord diagram\}

 \end{itemize}

\end{thm}

We consider sub chord diagram correponding to each $(a_i)$.

\begin{lem}\label{lem119}

Let ${SC}_{a_i}$ be the sub-chord diagram corresponding to the rectangle $(a_i)$.

\begin{enumerate}
 \item If two arcs of $(a_i)$ enter from the same side, $($left or right$)$, as shown in Figure \ref{fig:4-01ji} $\textcircled{\scriptsize1}$ and  $\textcircled{\scriptsize2}$, then any two chords in ${SC}_{a_i}$ certainly cross each other $($i.e.\ any two chards have an intersection$)$ as shown in Figure \ref{fig:4-01ji} $\textcircled{\scriptsize7}$.

 \item If one of two arcs enters from the left-hand side of $(a_i)$ and the other enters from the right-hand side as shown in Figure\ref{fig:4-01ji} $\textcircled{\scriptsize3}$, $\textcircled{\scriptsize4}$, $\textcircled{\scriptsize5}$ and $\textcircled{\scriptsize6}$, then there are no intersections in ${SC}_{a_i}$ as shown in Figure \ref{fig:4-01ji} $\textcircled{\scriptsize8}$. That is to say ${SC}_{a_i}$ is parallel.

\end{enumerate}

\end{lem}

\begin{figure}[htbp]
 \centering
 \includegraphics[width=6.5cm,clip]{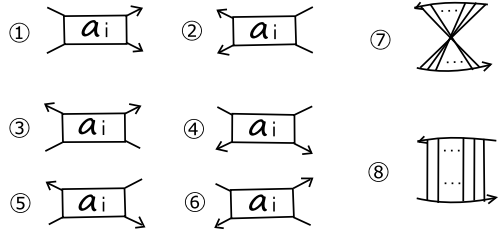}
 \caption{The orientations of two arcs in $(a_i)$ and the sub-chord diagram corresponding to $(a_i)$}
 \label{fig:4-01ji}
\end{figure}

\begin{proof}

First we name the crossings in $(a_i)$,\ $1, 2, \ldots, k$ from left to right. 

\begin{enumerate}

 \item If an arc enters from the left side, it passes the crossings 1, 2, $\ldots$ , k in order. Since the arc enters from the left side again, it passes the crossings in the same order. Therefore, the sub chord diagram corresponding to $(a_i)$ is as shown in Figure  \ref{fig:4-01-01ji}(1 ).

 \item On the otherhand, when the arcs enter from the different sides of $(a_i)$, the sub-chord diagram corresponding to $(a_i)$ is naturally as shown in Figure \ref{fig:4-01-01ji}( 2 ).

\end{enumerate}

\end{proof}

\begin{figure}[htbp]
 \centering
 \includegraphics[width=4cm,clip]{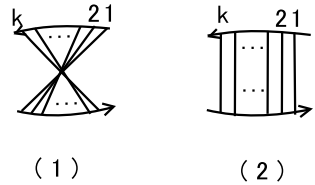}
 \caption{The sub-chord diagram corresponding to $(a_i)$}
 \label{fig:4-01-01ji}
\end{figure}

From the lemma above, we can consider a sub-chord diagram of a rectangle as one chord. That is to say, we can gather all chords in the sub-chord diagram corresponding to $(a_i)$ into one chord denoted by $\overline{a}_i$. Furthermore, in the case of $(1)$ in Lemma \ref{lem119}, we  name this chord \emph{I-chord} then represent it by a dotted line, while in the case of $(2)$ we name this chord \emph{P-chord} and represent it by a solid line. 

Moreover, we determine the chord diagram $CD_P$ corresponding to the diagram $D$. When $D$ is of type 1a, by thinking over the orientation of each rectangle, we can see that any $\overline {a}_{2i}$ for $1 \leq i \leq n$ is a P-chord, and any $\overline {a}_{2i-1}$ for $1 \leq i \leq n$ is an I-chord. If every $a_{2i-1}$, $1 \leq i \leq n$, is even, then we obtain the diagram as shown in Figure\ref{fig:4-02ji}(1), and also obtain the chord diagram as shown in Figure\ref{fig:4-02ji}(2). (For convenience, we represent a chord diagram not by a circle but by a quadrangle.)

\begin{figure}[htbp]
 \centering
 \includegraphics[width=12cm,clip]{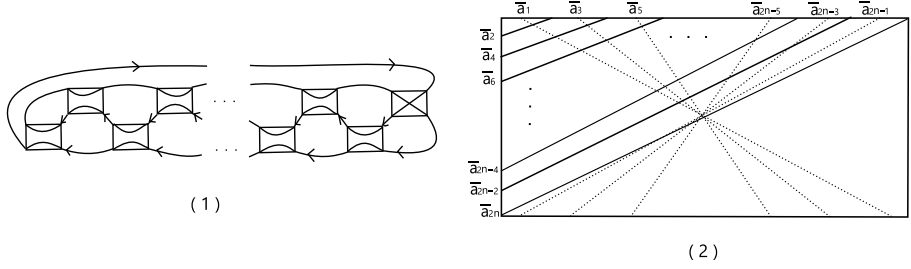}
 \caption{An oriented diagram $D$ of type 1a with every $a_{2i-1}$ even, and the chord diagram corresponding to $D$}
 \label{fig:4-02ji}
\end{figure}

Otherwise, the way to round a diagram depends on whether $a_{2i-1}$ is odd or even. So we rename the lower rectangles which consist of odd number half-twists (the number of these tangles is even),  $(b_1), (b_2), \ldots , (b_{2r})$ from left to right in the diagram. Moreover we also rename all upper rectangles as following:

\begin{itemize}

\item the upper rectangles on the left-hand side of $(b_1)$;  \ $(c_0^1), (c_0^2), \ldots, (c_0^{q_0})$\\

\item the upper rectangles between $(b_j)$ and $(b_{j+1})$;  \ $(c_j^1), (c_j^2), \ldots, (c_j^{q_j})$\\

\item the upper rectangles on the right-hand side of $(b_{2r})$;  \ $(c_{2r}^1), (c_{2r}^2), \ldots, (c_{2r}^{q_{2r}})$

\end{itemize}

Then we can obtain the sub chord diagram corresponding to the rectangles between $(b_j)$ and $(b_{j+1})$ as shown in Figure\ref{fig:4-03ji}, where $\overline{{c}_j^k}$ $(1 \leq k \leq q_j)$ consists of some parallel chords which correspond to the crossings in rectangle $(c_j^k)$.

\begin{figure}[htbp]
 \centering
 \includegraphics[width=12cm,clip]{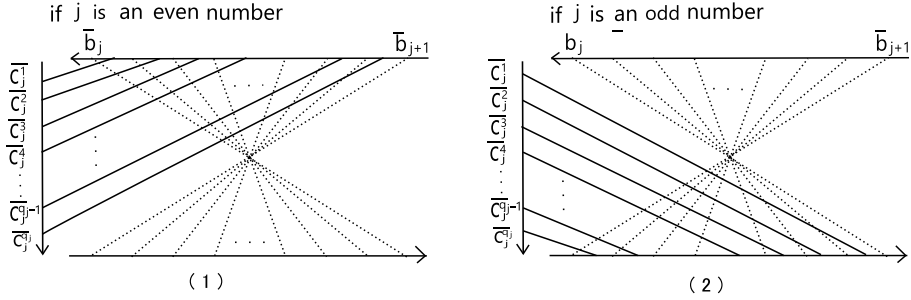}
 \caption{The sub chord diagram between $(b_i)$ and $(b_{j+1})$}
 \label{fig:4-03ji}
\end{figure}

Furthermore, any two P-chords in $\overline{c_j^1}, \overline{c_j^2}, \ldots , \overline{c_j^{q_j}}$ does not cross each other, so we can bundle them again into one P-chord. Now we represent them by a solid line. In other words, we consider that $c_j = c_j^1 + c_j^2 + \ldots + c_j^{q_j}$.
Since the chords which correspond to the lower tangles between  $(b_j)$ and $(b_{j+1})$ are all I-chords, by Theorem \ref{thm117} the number of chords which we can leave is at most only one. Besides, the I-chord which does not cross $\overline{c}_j$ is only $\overline{b}_j$ as shown in Figure \ref{fig:4-03ji}(1) or $\overline{b}_{j+1}$ as shown in Figure   \ref{fig:4-03ji}(2). So we only need to consider the chord diagram in which all I-chords between $\overline{b}_j$ and $\overline{b}_{j+1}$ were already delated as shown in Figure \ref{fig:4-04ji}.

\begin{figure}[htbp]
 \centering
 \includegraphics[width=8cm,clip]{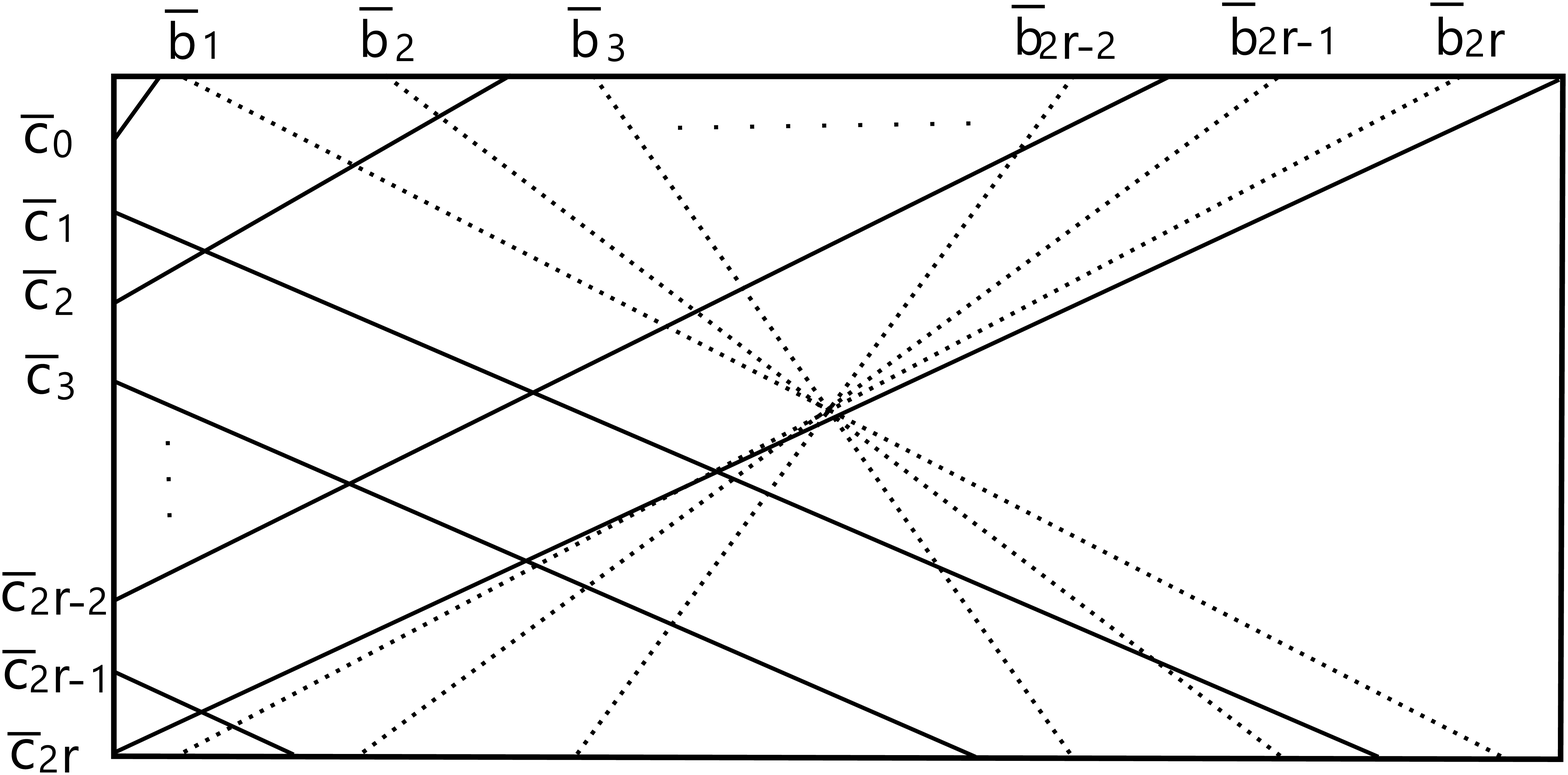}
 \caption{The chord diagram corresponding to the diagrams of type 1a.}
 \label{fig:4-04ji}
\end{figure}

In the case of the diagram of type 1b, we can consider in a similar fashion to type 1a. When every $a_{2i}$ is even, the diagram is as shown in Figure \ref{fig:4-04-01ji}(1), and the chord diagram is as shown in Figure \ref{fig:4-04-01ji}(2). Otherwise, the sub chord diagram between $(b_j)$ and $(b_{j+1})$ is also as shown in Figure \ref{fig:4-03ji}, and we can get the chord diagram as shown in Figure \ref{fig:4-06ji}. 

\begin{figure}[htbp]
 \centering
 \includegraphics[width=12cm,clip]{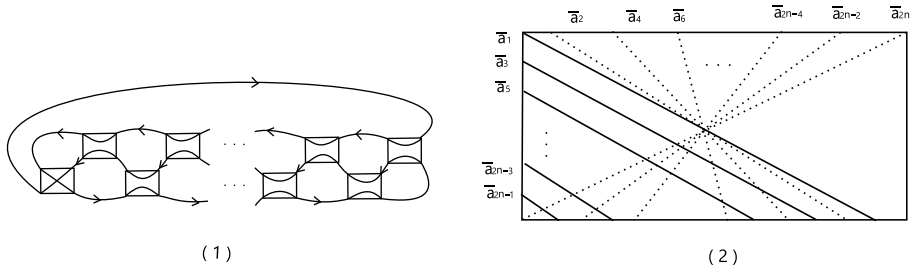}
 \caption{An oriented diagram $D$ of type 1b with every $a_{2i}$ even, and the chord diagram corresponding to $D$}
 \label{fig:4-04-01ji}
\end{figure}

\begin{figure}[htbp]
 \centering
 \includegraphics[width=8cm,clip]{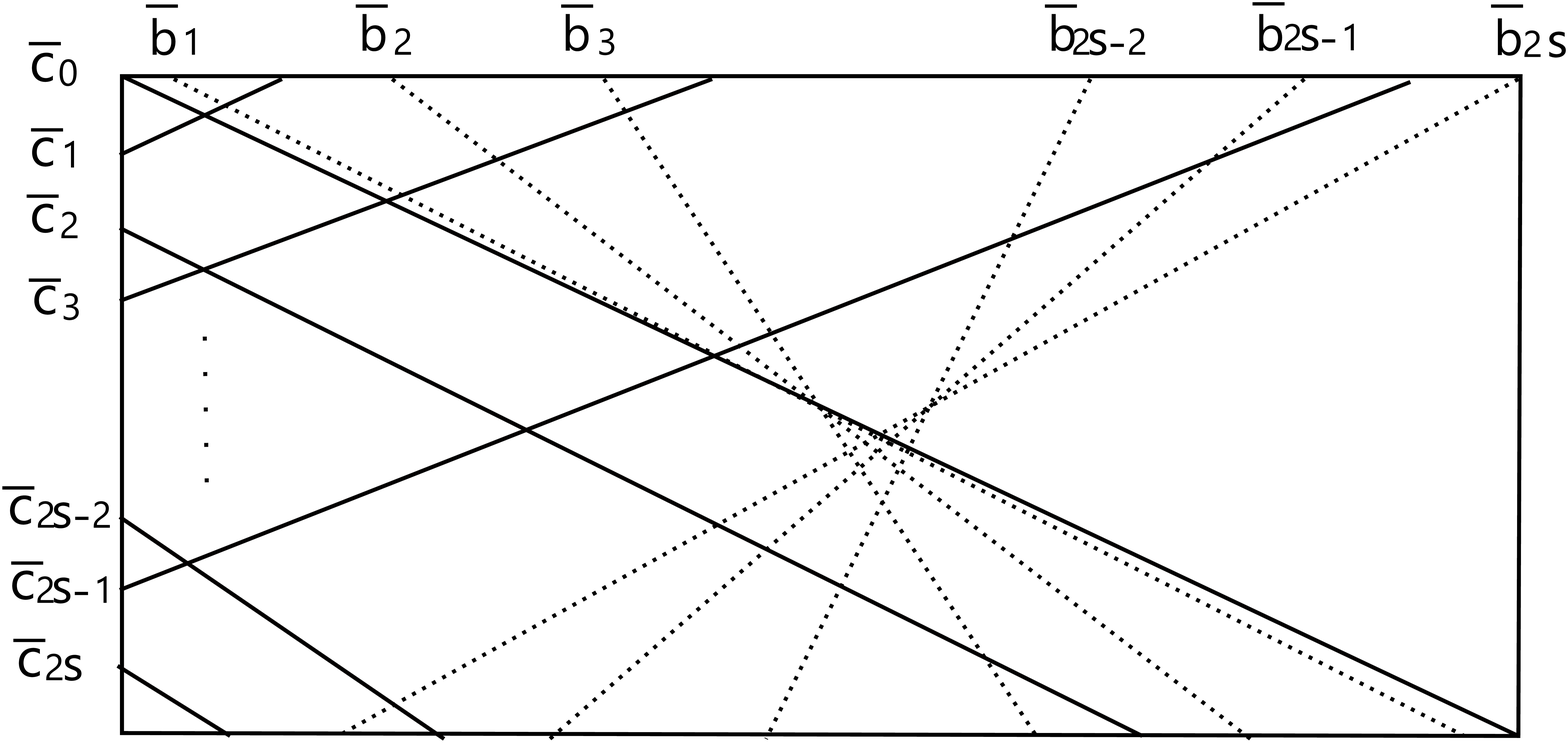}
 \caption{The chord diagram corresponding to the diagrams of type 1b.}
 \label{fig:4-06ji}
\end{figure}

On the other hand, in the case of the diagrams of type 2a and type 2b, when $j$ is an even number, the sub chord diagram is as shown in Figure \ref{fig:4-03ji}(1), and when $j$ is an odd number, the sub chord diagram is as shown in Figure \ref{fig:4-03ji}(2). Thus we can get the chord diagrams as shown in Figure \ref{fig:4-07ji} and Figure \ref{fig:4-08ji}.

\begin{figure}[htbp]
 \centering
 \includegraphics[width=8cm,clip]{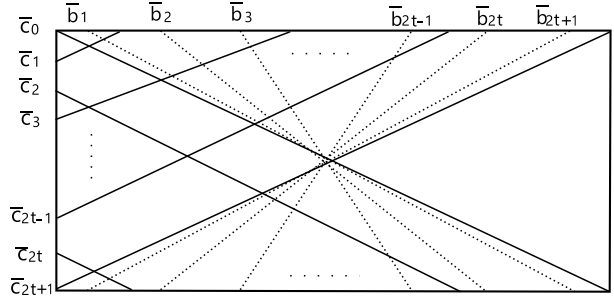}
 \caption{The chord diagram corresponding to the diagrams of type 2a.}
 \label{fig:4-07ji}
\end{figure}

\begin{figure}[htbp]
 \centering
 \includegraphics[width=8cm,clip]{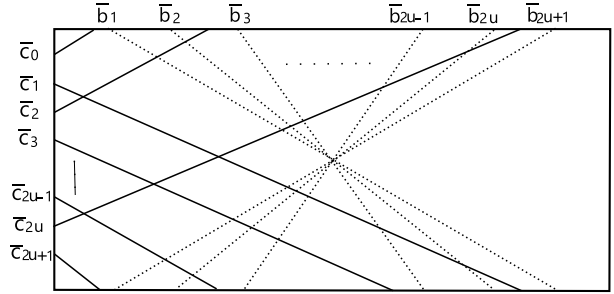}
 \caption{The chord diagram corresponding to the diagrams of type 2b.}
 \label{fig:4-08ji}
\end{figure}

In the condition above, we can obtain the theorem bellow.

\begin{thm}\label{thm118}

Let $D = D(a_1, a_2, \ldots , a_m)$ be a positive diagram or a negative diagram of a $2$-bridge knot such that $a_i > 0$ for all $i$ with $1 \leq i \leq m$. Then for the trivializing number of $D$, the following holds.

\begin{enumerate}

 \item When $D$ is of type 1a.

  \begin{enumerate}
  
    \item If every $a_{2i-1}$ is even, then $tr(D) = \sum_{i=1}^n a_{2i-1}$.

    \item Otherwise, 

$tr(D) = min\begin{Bmatrix} 
           \sum_{i=1}^n a_{2i-1} + \sum_{j=1}^r c_{2j-1}\\
           \sum_{i=1}^n a_{2i-1} + \sum_{j=1}^p c_{2j-1} + \sum_{j=p+1}^r c_{2j} - 1   
            \end{Bmatrix}$
           
   \end{enumerate}
        
  \item When $D$ is of type 1b.

    \begin{enumerate}

    \item If every $a_{2i}$ is even, then $tr(D) = \sum_{i=1}^n a_{2i}$.

    \item Otherwise, 

$tr(D) = min\begin{Bmatrix}
             \sum_{i=1}^n a_{2i} + \sum_{j=1}^s c_{2j-1}\\
             \sum_{i=1}^n a_{2i} + \sum_{j=0}^p c_{2j} + \sum_{j=p+2}^s c_{2j-1} - 1
            \end{Bmatrix}$

    \end{enumerate}
 
 \item When $D$ is of type 2a.
 
$tr(D) = min\begin{Bmatrix}
             \sum_{i=1}^n a_{2i} + \sum_{j=0}^t c_{2j+1}\\
             \sum_{i=1}^n a_{2i} + \sum_{j=0}^p c_{2j} + \sum_{j=p+1}^t c_{2j+1}
       
            \end{Bmatrix}$

  \item When $D$ is of type 2b. 

$tr(D) = min\{  \sum_{i=1}^n a_{2i+1} + \sum_{j=0}^p c_{2j+1} + \sum_{j=p+2}^u  c_{2j} - 1 \}$

\end{enumerate}

\end{thm}

\begin{pr}

\begin{enumerate}

 \item When $D$ is of type 1a.

 \begin{enumerate}

\item If every $a_{2i-1}$ is even, then we have the chord diagram as shown in Figure\ref{fig:4-02ji}(2). Since any two I-chords in this chord diagram cross each other, we can leave at most only one I-chord when we attempt to gain a trivial chord diagram. Moreover every two chords in any I-chord also cross each other. This means that the number of the chords corresponding to the crossings in lower rectangles which we can leave is at most only one.

In addition, the P-chords corresponding to the crossings in upper rectangles are all parallel and any P-chord crosses at least one I-chord. Hence the minimal number of the chords which we must delate in order to get a trivial chord diagram is the number of all the chords which correspond to the crossings in lower rectangles. Therefore, we have the following:

$tr(D) = \sum_{i=1}^n a_{2i-1}$

 \item Otherwise, the chord diagram as shown in Figure \ref{fig:4-04ji}, and we can see the P-chord represented by $\overline{c}_{2r}$ crosses all P-chords represented by $\overline{c}_{2j-1}$ \ $(1 \leq j \leq r)$ \ and all I-chords represented by $\overline{b}_k$ \ $(1 \leq k \leq r)$. (This means that $\overline{c}_{2r}$ crosses all chords corresponding to the crossings in lower rectangles). So if we leave $\overline{c}_{2r}$ then we must delete all these chords which cross $\overline{c}_{2r}$. That is to say the number of chords we must delete is $\sum_{i=1}^n a_{2i-1} + \sum_{j=1}^r c_{2j-1}$. See Figure \ref{fig:4-05ji}(1).

When we delete $\overline{c}_{2r}$, we can leave all chords in P-chord $\overline{c}_{2r-1}$ and only one chord in I-chord $\overline{b}_{2r-1}$. So the number of chords we need to delete is $\sum_{i=1}^n a_{2i-1} + \sum_{j=1}^{r-1} c_{2j-1} + c_{2r} - 1$. See Figure \ref{fig:4-05ji}(2). 

\begin{figure}[htbp]
 \centering
 \includegraphics[width=13cm,clip]{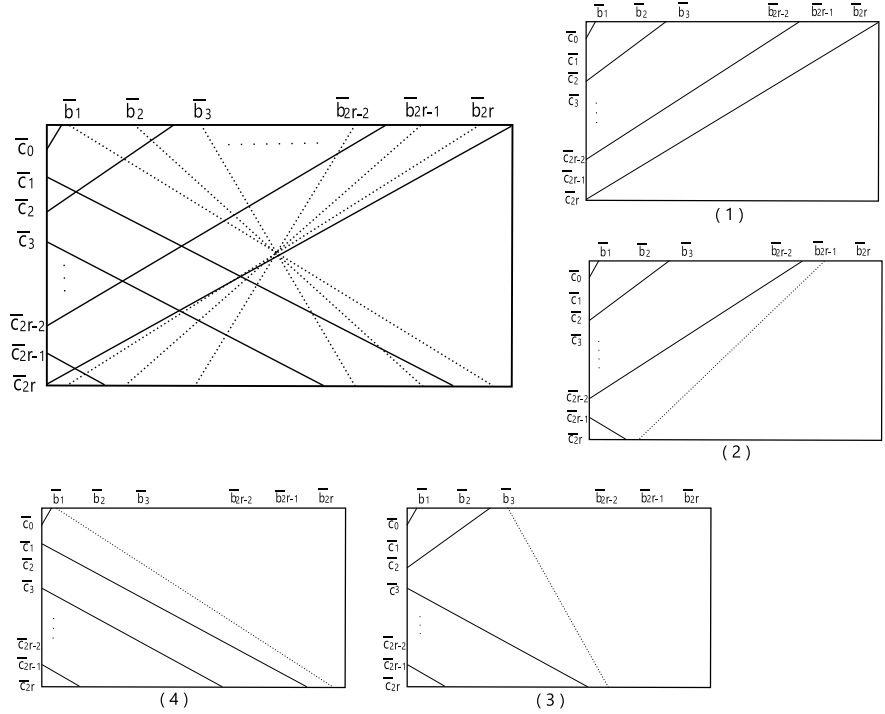}
 \caption{The operation of deleting some P-chords}
 \label{fig:4-05ji}
\end{figure}

Next we attempt to delete the P-chords which correspond to $\overline{c}_{2j}$  $(1 \leq j \leq r)$ step by step in the way as following: \{$\overline{c}_{2r}\} \rightarrow \{\overline{c}_{2r}, \overline{c}_{2r-2}\} \rightarrow \{\overline{c}_{2r},\overline{c}_{2r-2},\overline{c}_{2r-4}\} \rightarrow \cdots \rightarrow \{\overline{c}_{2r}, \overline{c}_{2r-2}, \cdots, \overline{c}_2 \} $. 

By these operations we can also get a trivial chord diagram even if we leave the P-chords which correspond to $\overline{c}_{2j-1}$  $(1 \leq j \leq r)$ step by step in the way as following: \{$\overline{c}_{2r-1}\} \rightarrow \{\overline{c}_{2r-1},\overline{c}_{2r-3}\} \rightarrow \{\overline{c}_{2r-1},\overline{c}_{2r-3},\overline{c}_{2r-5}\} \rightarrow \cdots \rightarrow \{\overline{c}_{2r-1},\overline{c}_{2r-3} \cdots \overline{c}_1\}$. 

There is an one-to-one correlation between these two operations. Consequently, the minimum of these numbers is the trivializing number of the diagram, and the following holds.

$tr(D) = min\begin{Bmatrix}
            \sum_{i=1}^n a_{2i-1} + \sum_{j=1}^r c_{2j-1}\\
            \sum_{i=1}^n a_{2i-1} + \sum_{j=1}^p c_{2j-1} + \sum_{j=p+1}^r c_{2j} - 1
            \end{Bmatrix}$ \ 
\vspace{0.3cm}

  \end{enumerate}

 \item When $D$ is of type 1b.
  
\begin{enumerate}

 \item If every $a_{2i}$ is even, then we can consider in a similar fashion to type 1a and can easily see $tr(D) = \sum_{i=1}^n a_{2i}$.

 \item Otherwise, from the chord diagram as shown in Figure\ref{fig:4-06ji},  we know the P-chord $\overline{c}_{2r}$ in Figure\ref{fig:4-04ji} is replaced by $\overline{c}_0$ in Figure \ref{fig:4-06ji}. In this case, if we delete the P-chords represented by $\overline{c}_{2j}$ \ $(0 \leq j \leq s)$ step by step in the way \{$\overline{c}_0\} \rightarrow \{\overline{c}_0,\overline{c}_2\} \rightarrow \{\overline{c}_0,\overline{c}_2,\overline{c}_4\} \ldots $, then we can leave \{$\overline{c}_1\} \rightarrow \{\overline{c}_1,\overline{c}_3\} \rightarrow \{\overline{c}_1,\overline{c}_3,\overline{c}_5\} \ldots $, by way of compensation. Thus the following holds.

$tr(D) = min\begin{Bmatrix}
             \sum_{i=1}^n a_{2i} + \sum_{j=1}^s c_{2j-1}\\
             \sum_{i=1}^n a_{2i} + \sum_{j=0}^p c_{2j} + \sum_{j=p+2}^s c_{2j-1} - 1
            \end{Bmatrix}$ \

\end{enumerate}

\item When $D$ is of type 2a. In this case, the chord diagram is as shown in Figure \ref{fig:4-07ji}, and we see that every I-chord which corresponds to $(b_j)$ $(1 \leq j \leq 2t+1)$ necessarily crosses two P-chords which correspond to $(c_0)$ and $(c_{2t+1})$. So we can leave none of these I-chords unless we delete both $\overline{c}_0$ and $\overline{c}_{2t+1}$. In addition, we consider the relation of P-chords which correspond to $(c_j)$ \ $(1 \leq j \leq 2t+1)$. If we leave every $\overline{c}_{2j}$ \ $(0 \leq j \leq t)$, then we must delete every $\overline{c}_{2j+1}$ \ $(0 \leq j \leq t)$. Hence the number of chords which we need to delete is the following:

$\sum_{i=1}^n a_{2i} + \sum_{j=0}^t c_{2j+1}$.

Furthermore there exists a relation in the P-chords in this chord diagram. That is, if we delete $\overline{c}_0$ then we can leave $\overline{c}_1$, if we delete \{$\overline{c}_0, \overline{c}_2$\} then we can leave \{$\overline{c}_1, \overline{c}_3$\}, and so on. Because of this, the following holds. 

$tr(D) = min\begin{Bmatrix}
             \sum_{i=1}^n a_{2i} + \sum_{j=0}^t c_{2j+1}\\
             \sum_{i=1}^n a_{2i} + \sum_{j=0}^p c_{2j} + \sum_{j=p+1}^t c_{2j+1}
            
            \end{Bmatrix}$

\item When $D$ is of type 2b, the chord diagram is as shown in Figure \ref{fig:4-08ji}. In this chord diagram, $\overline{c}_0$ and $\overline{c}_{2u+1}$ dose not cross each other. Moreover they dose not cross any other P-chord or I-chord. Therefore, we can leave both $\overline{c}_0$ and $\overline{c}_{2u+1}$. However for I-chords $\overline{b}_j$ \ $(1 \leq j \leq 2u+1)$, any two of them cross each other, so we can leave at most only one I-chord among \{$\overline{b}_j$\}. Thus, if we leave all P-chords corresponding to $(c_{2k})$ \ $(1 \leq k \leq u)$, we must delete all P-chords corresponding to $(c_{2k-1})$ \ $(1 \leq k \leq u)$. Hence the number of all chords which we must delete is $\sum_{i=1}^n a_{2i+1} + \sum_{j=0}^{u-1} c_{2j+1} - 1$.

Besides, if we orderly delete some P-chords step by step such as \{$\overline{c}_{2u}\} \rightarrow \{\overline{c}_{2u},\overline{c}_{2u-2}\} \rightarrow \{\overline{c}_{2u},\overline{c}_{2u-2},\overline{c}_{2u-4}\} \cdots $ , we can leave other P-chords such as  \{$\overline{c}_{2u-1}\} \rightarrow \{\overline{c}_{2u-1},\overline{c}_{2u-3}\} \rightarrow \{\overline{c}_{2u-1},\overline{c}_{2u-3},\overline{c}_{2u-5}\} \cdots $ by way of compensation. Finally the following holds.

$tr(D) = min\{ \sum_{i=1}^n a_{2i+1} + \sum_{j=0}^p c_{2j+1} + \sum_{j=p+2}^u  c_{2j} - 1 \}$ 

\end{enumerate}

We have just completed the proof of Main Theorem.

\qed\end{pr}

\section{The relation between trivializing number and unknotting number}

In this section we study the relation between the trivializing number and the unknotting number. The definitions of the unknotting number of a diagram and the unknotting number of a knot are the following:

\begin{deff}\label{deff120}
The \emph{unknotting number} of a diagram $D$, denoted by $\mathrm{u}(D)$, is the minimal number of crossings of $D$ whose over/under information should be changed for getting a diagram of a trivial knot.
\end{deff}

\begin{deff}\label{deff121}
The \emph{unknotting number} of a knot $K$, denoted by $\mathrm{u}(K)$, is the minimum of $\mathrm{u}(D)$, where the minimum is taken over all diagrams $D$ of $K$.
\end{deff}

There is a relation between the unknotting number and the signature. (About the signature there is a detailed explanation in \cite{Mu}). The signature is an invariant of knots, and in general the following holds.

\begin{thm}[\cite{Mu}]\label{thm122}

$\frac{1}{2}{|\sigma(D)|} = \frac{1}{2}{|\sigma(K)|} \leq u(K) \leq u(D)$

\end{thm}

In addition for a alternating diagram, it is known that $\sigma(D) = -w(D) / 2 + (W - B) / 2$ (\cite{Tr}), where $w(D)$ is the sum of local writhes of all crossings, $B$ is the number of domains colored with a grayish color when we give checkerboard coloring as shown in Figure \ref{fig:5-01ji}, and $W$ is the number of domains which are not colored. For example, in the case as shown in Figure \ref{fig:5-01ji}, the number of $+$ crossings is $2$, and that of $-$ crossings is $4$, then $\sigma(D) = 2 + (-4) = -2$, and $W = 5$, $B = 3$. Therefore we can get $\sigma(D) = \sigma(K) = -(-2) / 2 + (5 - 3) /2 = 2$, and $|\sigma(D)| / 2 = 1 \leq u(K) \leq u(D)$. In actually, we can obtain a diagram of a trivial knot with one crossing change, hence $u(D) = u(K) = 1$.

\begin{figure}[htbp]
 \centering
 \includegraphics[width=9cm,clip]{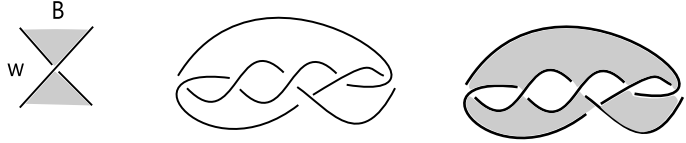}
 \caption{An example of the checkerboard coloring and local writhes}
 \label{fig:5-01ji}
\end{figure}

About the relation between the trivializing number and the unknotting number, it is known that $2u(D) \leq tr(D)$ \  $(2u(K) \leq tr(K))$ holds in general. However, particularly for positive knots, there exists a conjecture that $2u(K) = tr(K)$ (\cite{Ha}). And as the partial positive answer of this, we have the next corollary to Theorem \ref{thm118} and Theorem \ref{thm122}.

\begin{cor}\label{cor123}

Let $K$ be a positive 2-bridge knot and has a diagram $D$ = $C(a_1, a_2, \ldots , a_{2n})$ \ $a_i > 0$ for any 
$i$ \ $(1 \leq i \leq 2n)$. If $a_{2i-1}$ is even for any $i$ \ $(1 \leq i \leq n)$, or $a_{2i}$ is even for any $i$ \ $(1 \leq i \leq n)$, then $2u(K) = tr(K)$.
 
\end{cor}

\begin{proof}\label{pr124}
First we prove the case where any $a_{2i-1}$ is an even number. In this case, $D$ is a minimal diagram of $K$, so by Theorem \ref{thm114}, $D$ is an positive and alternating diagram. Besides, by the Proposition \ref{prop116},  $a_{2n}$ must be an odd number and other $a_{2i} \ (1 \leq i \leq n-1)$ \ are necessarily all even numbers. Moreover, the sign of any crossing is $+$, thereby $w(D) = \sum_{i=1}^{2n} a_i$. The checkerboard coloring is like as shown in Figure \ref{fig:5-02ji}, and we know $W = \sum_{i=1}^n a_{2i} + 1$, $B = \sum_{i=1}^n a_{2i-1} + 1$. 

\begin{figure}[htbp]
 \centering
 \includegraphics[width=8cm,clip]{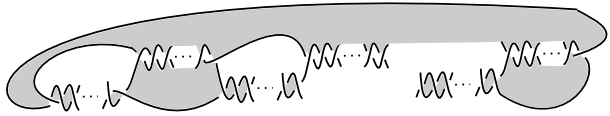}
 \caption{The checkerboard coloring of a 2-bridge knot diagram in which any $a_{2i-1}$ is an even number}
 \label{fig:5-02ji}
\end{figure}

Therefore, next equality holds.

 $\sigma(D) = -\frac{1}{2}(w(D)) + \frac{1}{2}(W - B) = \frac{1}{2}( -\sum_{i=1}^{2n} a_i + \sum_{i=1}^n a_{2i} - \sum_{i=1}^n a_{2i-1}) = -\sum_{i=1}^n a_{2i-1}$

Furthermore by Theorem \ref{thm122}, we can see $(|\sigma(D)|)/2 = (\sum_{i=1}^n a_{2i-1})/2 \leq u(K) = u(D)$. In actually as shown in Figure \ref{fig:5-03ji}, we can obtain a trivial diagram by some crossing changes of the crossings which correspondent to lower tangles, and the number of these crossing changes is $(\sum_{i=1}^n a_{2i-1})/2$. Hence $u(D) = u(K) = (\sum_{i=1}^{n} a_{2i-1})/2$. 

\begin{figure}[htbp]
 \centering
 \includegraphics[width=8cm,clip]{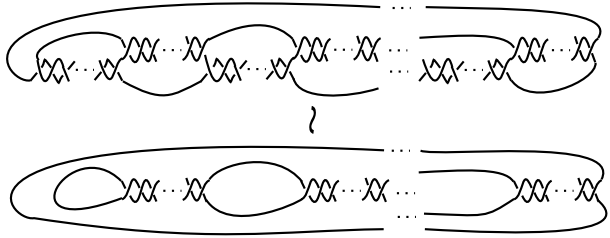}
 \caption{By some crossing changes, we can obtain a trivial knot diagram.}
 \label{fig:5-03ji}
\end{figure}

Finally we can get the inequality $2u(K) \leq tr(K) \leq tr(D)$ and  the equality $2u(K) = tr(D)$. Thus $2u(K) = tr(K)$ holds.

In the case that any $a_{2i}$ is an even number, we can also gain this equality in a similar fashion.

\end{proof}

This result is for the special case of positive 2-bridge knots. So whether \ $2u(D) = tr(D)$ holds for any minimal diagram of positive 2-bridge knots or not, and whether \ $2u(K) = tr(K)$ holds or not, these questions are our theme of the future.

\section{Positive pretzel knot}

For positive pretzel knots we can get the following:

\begin{thm}\label{thm125}

Let $K$ be a pretzel knot $P(p_1, p_2, \ldots , p_{2n})$ \ $p_i > 0$ for any $i$ $(1 \leq i \leq 2n)$, $p_{2n}$ is even and other $p_i$s are all odd $(1 \leq i \leq 2n-1)$, then the following holds.

$tr(K) = 2u(K) = \sum_{i=1}^{2n} p_i - 2n + 1$

\end{thm}

\begin{proof}\label{pr126}

A diagram $D$ of knot $K$ is as shown in Figure \ref{fig:6-01ji}, and we know this diagram is positive and alternating. Besides the sub-chord diagram which corresponds to each $\overline{p}_i$ is an I-chord. So the chord diagram of $D$ is as shown in Figure \ref{fig:6-02ji}.

\begin{figure}[htbp]
 \centering 
 \includegraphics[width=6.5cm, clip]{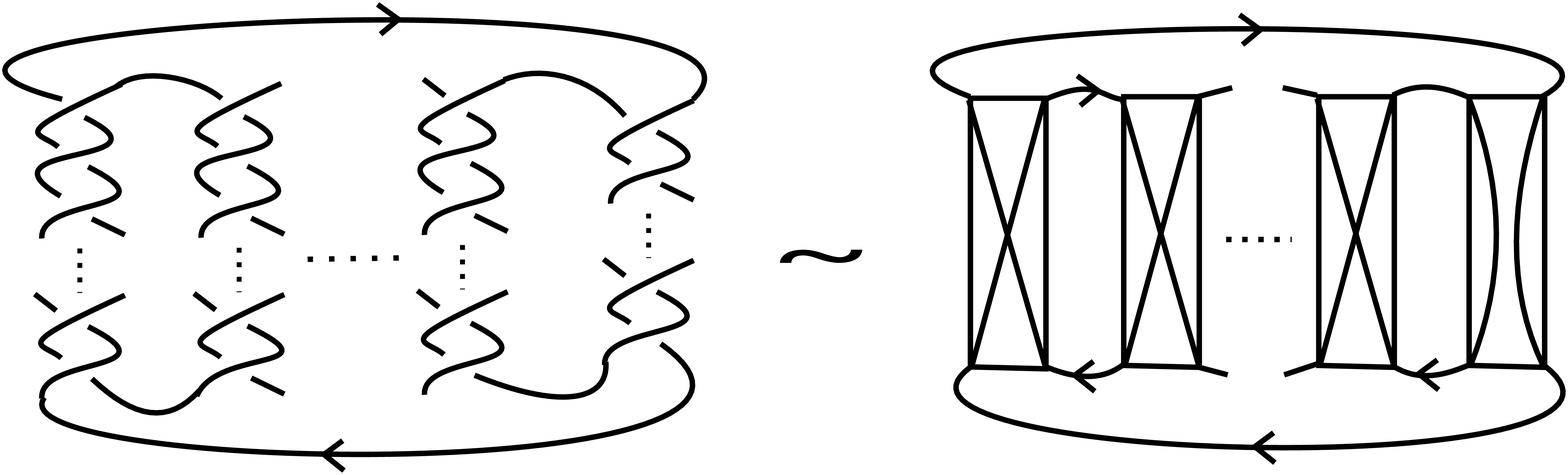}
 \caption{The standard diagram $D$ of $K$}
 \label{fig:6-01ji}
 \end{figure}

\begin{figure}[htbp]
 \centering 
 \includegraphics[width=9cm, clip]{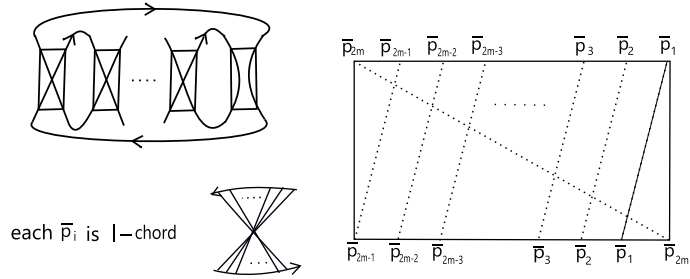}
\caption{The chord diagram of $K$}
 \label{fig:6-02ji}
 \end{figure}

Then we can easily obtain the trivializing number of $D$. Namely, $tr(D) = \sum_{i=1}^{2n} p_i - 2n + 1$.

Furthermore, by the checkerboard coloring as shown in Figure \ref{fig:6-03ji}, the signature of $K$ is the following:

 $\sigma(K) = \sigma(D) = -\frac{1}{2}w(D) + \frac{1}{2}(W - B) = - (\sum_{i=1}^{2n} p_i - 2n + 1)$

By the inequality $|\sigma(K)| \leq 2u(K) \leq tr(K) \leq tr(D)$, we can conclude $tr(K) = 2u(K)$.

This completes the proof.

\begin{figure}[htbp]
 \centering 
 \includegraphics[width=3.5cm, clip]{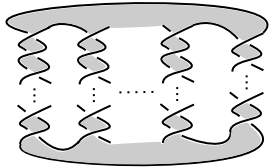}
 \caption{An example of checkerboard coloring}
 \label{fig:6-03ji}
 \end{figure}

\end{proof}

\end{document}